\documentclass[11pt]{amsart}

\usepackage[T1]{fontenc}
\usepackage{lmodern}
\usepackage{microtype}
\usepackage{amsmath,amssymb,amsfonts,amsthm,mathtools,mathrsfs}
\usepackage{enumitem}
\usepackage{aliascnt}
\usepackage{xcolor}
\usepackage[colorlinks=true,linkcolor=blue,citecolor=blue,urlcolor=blue]{hyperref}
\usepackage[capitalize,noabbrev]{cleveref}

\theoremstyle{plain}
\newtheorem{theorem}{Theorem}[section]

\newaliascnt{proposition}{theorem}
\newtheorem{proposition}[proposition]{Proposition}
\aliascntresetthe{proposition}

\newaliascnt{lemma}{theorem}
\newtheorem{lemma}[lemma]{Lemma}
\aliascntresetthe{lemma}

\newaliascnt{corollary}{theorem}
\newtheorem{corollary}[corollary]{Corollary}
\aliascntresetthe{corollary}

\newaliascnt{conjecture}{theorem}

\aliascntresetthe{conjecture}

\theoremstyle{definition}
\newaliascnt{definition}{theorem}
\newtheorem{definition}[definition]{Definition}
\aliascntresetthe{definition}

\newaliascnt{convention}{theorem}
\newtheorem{convention}[convention]{Convention}
\aliascntresetthe{convention}

\newaliascnt{example}{theorem}

\aliascntresetthe{example}

\theoremstyle{remark}
\newaliascnt{remark}{theorem}
\newtheorem{remark}[remark]{Remark}
\aliascntresetthe{remark}

\crefname{conjecture}{Conjecture}{Conjectures}
\Crefname{conjecture}{Conjecture}{Conjectures}
\crefname{proposition}{Proposition}{Propositions}
\Crefname{proposition}{Proposition}{Propositions}
\crefname{corollary}{Corollary}{Corollaries}
\Crefname{corollary}{Corollary}{Corollaries}
\crefname{lemma}{Lemma}{Lemmas}
\Crefname{lemma}{Lemma}{Lemmas}
\crefname{definition}{Definition}{Definitions}
\Crefname{definition}{Definition}{Definitions}
\crefname{convention}{Convention}{Conventions}
\Crefname{convention}{Convention}{Conventions}
\crefname{example}{Example}{Examples}
\Crefname{example}{Example}{Examples}
\crefname{remark}{Remark}{Remarks}
\Crefname{remark}{Remark}{Remarks}

\newcommand{\C}{\mathbb C}
\newcommand{\Pj}{\mathbb P}
\newcommand{\N}{\mathbb N}
\newcommand{\OO}{\mathcal O}
\newcommand{\W}{\mathcal W}
\newcommand{\F}{\mathcal F}

\newcommand{\Acal}{\mathcal A}
\newcommand{\Sym}{\operatorname{Sym}}
\newcommand{\Sol}{\operatorname{Sol}}
\newcommand{\rank}{\operatorname{rank}}

\newcommand{\dd}{\operatorname{d}}
\newcommand{\boxt}{\boxtimes}
\newcommand{\CNF}{N^*_{\F}}
\newcommand{\NF}{N_{\F}}
\newcommand{\jet}{\mathcal J}
\newcommand{\Der}{\operatorname{Der}}

\newcommand{\Hom}{\operatorname{Hom}}
\newcommand{\Ann}{\operatorname{Ann}}

\newcommand{\Bott}{\operatorname{Bott}}

\title[Abelian relations are Liouvillian]
{Abelian relations are Liouvillian}

\author{Gabriel Fazoli}
\address{IMECC - UNICAMP,
Rua S\'ergio Buarque de Holanda, 651,
13083-859 Campinas, SP, Brazil}
\email{gabrielfazoli@gmail.com}

\author{Jorge Vit\'orio Pereira}
\address{IMPA, Estrada Dona Castorina, 110, 22460-320 Rio de Janeiro, RJ, Brazil}
\email{jvp@impa.br}

\keywords{Web geometry, abelian relations, Liouvillian functions,
holomorphic foliations, differential modules}
\subjclass[2020]{Primary 53A60; Secondary 12H05,14C21}

\date{}

\begin{document}

\begin{abstract}
We prove that every germ of an abelian relation of a codimension-one web on a complex projective manifold is defined over a Liouvillian extension of the function field. 
\end{abstract}

\maketitle
\section{Introduction}

The systematic study of \emph{functional equations} goes back at least to
Cauchy's \emph{Cours d'analyse} of 1821, where he considered the problem of
determining a function from a prescribed property that it must satisfy (see
\cite[Chapter~5]{Cauchy1821}). A basic prototype is the equation
\[
   F(xy)=F(x)+F(y),
\]
whose  continuous solutions on $\mathbb{R}_{>0}$ are the multiples
of the logarithm. One step beyond the
logarithm sits the \emph{dilogarithm}
$\operatorname{Li}_2(z)=\sum_{n\geq1}z^n/n^2$, and with it one of the most
celebrated functional equations of analysis: Abel's five-term equation. In
terms of the Rogers dilogarithm
$R(t)=\operatorname{Li}_2(t)+\frac12\log(t)\log(1-t)$, it reads
\begin{equation}
\label{eq:intro-five-term}
 R(x)-R(y)-R\left(\frac{x}{y}\right)
 -R\left(\frac{1-y}{1-x}\right)
 +R\left(\frac{x(1-y)}{y(1-x)}\right)= - \frac{\pi^2}{6},
\end{equation}
for $0<x<y<1$. Far from being a curiosity, the five-term equation and its
higher analogues permeate mathematics, from algebraic $K$-theory to the
special values of $L$-functions; see \cite{Oesterle1993,Zagier2007}.

Equations of this kind admit a geometric interpretation. Given functions
$U_1,\ldots,U_k$ of $n$ variables,
one may ask for functions
$F_1,\ldots,F_k$ of one variable satisfying
\begin{equation}
\label{eq:intro-functional-equation}
   F_1(U_1(x_1,\ldots,x_n))+\cdots+F_k(U_k(x_1,\ldots,x_n))=0.
\end{equation} 
Each $U_i$ defines a \emph{foliation} $\F_i$ by its level sets, and their
superposition defines a \emph{web}
\[
   \W=\F_1\boxt\cdots\boxt\F_k.
\]
Differentiating \eqref{eq:intro-functional-equation} gives
\[
   \dd\bigl(F_1(U_1)\bigr)+\cdots+ \dd\bigl(F_k(U_k)\bigr)=0,
\]
a sum of closed one-forms, each defining the corresponding foliation. Such
a tuple of one-forms is called an \emph{abelian relation} of the web, and,
conversely, the Poincar\'e Lemma shows that, locally and up to constants,
every abelian relation arises from a functional equation of the form
\eqref{eq:intro-functional-equation}. Already in 1823, Abel introduced a
method based on successive differentiations and eliminations to study such
equations \cite{Abel1823}; see \cite{Pirio2004,Pirio2006Selecta} for its
relation with web geometry. In this language, the five-term equation
\eqref{eq:intro-five-term} is an abelian relation of \emph{Bol's web}
\[
 \mathcal B_5=
 \W\left(x,\ y,\ \frac{x}{y},\ \frac{1-y}{1-x},\
 \frac{x(1-y)}{y(1-x)}\right),
\]
the superposition of the four pencils of lines through the base points of
a pencil of conics with the pencil of conics itself.

Webs and their abelian relations were studied intensively by Blaschke's
school in the 1930s \cite{BlaschkeBol1938}. The space of abelian relations
of a $k$-web on a surface has dimension at most $(k-1)(k-2)/2$, and the
webs attaining this bound are tightly linked to algebraic geometry: by
Abel's theorem, the $k$-web of tangent lines dual to a plane curve of
degree $k$ has maximal rank, its abelian relations being avatars of the
abelian differentials of the curve. Chern and Griffiths revived the
subject in the 1970s, connecting webs of maximal rank in higher dimensions
with Abel's theorem and the algebraization problem
\cite{ChernGriffiths1978}; see Beauville's report \cite{Beauville1980}.
For codimension-one webs in general position on $\C^n$, $n\geq3$,
the algebraization of webs of maximal rank with at least $2n$
foliations was established by Tr\'epreau \cite{Trepreau2006};
we refer to the Bourbaki seminar \cite{PereiraBourbaki2008}
for an account of this result and its history.
On surfaces, however, algebraization fails: Bol's web $\mathcal B_5$ has
maximal rank $6$ and is not algebraizable, precisely because of the
dilogarithmic relation \eqref{eq:intro-five-term} \cite{Bol1936}. Webs of
maximal rank which are not algebraizable are called \emph{exceptional},
and, paraphrasing  \cite[Section~6.2.1]{PereiraPirio2015}, their abelian
relations remain mysterious: even for a planar web defined by rational
functions, it is not known, in general, what kind of transcendency its
abelian relations can exhibit.

This paper answers the qualitative form of that question for
codimension-one webs on projective varieties. Recall that a differential
field extension is \emph{Liouvillian} if it is obtained through a tower of
algebraic extensions, adjunctions of primitives, and adjunctions of
exponentials of primitives: Liouvillian functions are exactly those
reachable from rational functions by the classical operations of
elementary calculus. In \cite[Conjecture~5.1]{MarinPereiraPirio2006},
it is conjectured that the abelian relations of a
web are defined over a Liouvillian extension of its field of definition;
the question is raised again as \cite[Problem~1]{PereiraPirio2015}.
Apart from hexagonal $3$-webs, whose unique abelian relation is
classically Liouvillian, two cases were known. For algebraic webs, the
abelian relations are given by abelian integrals, by Abel's theorem, and
are therefore Liouvillian; see \cite[Chapter~3]{PereiraPirio2015}. For
webs admitting a transverse infinitesimal automorphism defined over the
field of definition, the conjecture was established in
\cite{MarinPereiraPirio2006}, see also
\cite[Proposition~6.2.1]{PereiraPirio2015}.  Our main result confirms the
conjecture for codimension-one webs on smooth projective varieties.

\begin{theorem}
\label{thm:main}
Let $\W$ be a codimension-one $k$-web on a smooth projective variety $X$ of dimension at least two. Every germ of an abelian relation of $\W$ at a general regular point is defined over a Liouvillian extension of $\C(X)$.
\end{theorem}

We stress that no maximal-rank or algebraizability assumption is imposed
on the web. The conjecture of \cite{MarinPereiraPirio2006} remains open
for germs of webs defined over arbitrary complex differential fields.

\subsection*{Structure of the paper}
\Cref{sec:liouvillian-background} collects the background on Liouvillian
extensions, in arbitrary dimension. \Cref{sec:webs-abelian-relations}
introduces webs and abelian relations on smooth projective varieties
and reduces the problem to
proving that each component of a functional abelian relation
$u_1+\cdots+u_k=0$ of a decomposed web $\W=\F_1\boxt\cdots\boxt\F_k$ is
Liouvillian. The starting point of the
proof, developed in
\Cref{sec:rational-transverse-equation,sec:transverse-affine-structure},
is that all the components attached to a foliation $\F_i$ satisfy a
rational differential equation in the transverse direction, which forces
$\F_i$ to be transversely affine. When $\F_i$ has no rational first
integral, this structure is rigid enough to turn every component into an
explicit Liouvillian expression
(\Cref{sec:no-rational-first-integrals}).  
When $\F_i$ is defined by a rational fibration, its components
generate differential modules which descend to the base curve.
\Cref{sec:component-modules,sec:resonance-descent} studies these
modules and the resonances imposed by the abelian relation. 
\Cref{sec:proof-main-theorem} combines these results with the
analysis of foliations without rational first integrals to obtain,
after a finite algebraic extension, triangular systems for the
components, which are then solved by quadratures and exponentials,
completing the proof of \Cref{thm:main}.

\subsection*{Acknowledgements} The authors express their appreciation for the financial support provided by CAPES/COFECUB and CNPq Projeto Universal 408687/2023-1 ``Geometria das Equa\c{c}\~oes Diferenciais Alg\'ebricas''. Fazoli acknowledges the support from CAPES (Grant number 88887.184306/2025-00), and FAPESP (Grant number 2026/14134-9). Pereira acknowledges the support from CNPq (Grant number 304690/2023-6), and FAPERJ (Grant number E26/200.550/2023).

\subsection*{Use of Large Language Models}
During the development and preparation of this article, the authors used ChatGPT (OpenAI; GPT-5.6 Sol) and Claude (Anthropic; Claude Fable 5), as interactive tools to explore proof strategies, and to assist with drafting and revision. The authors assume full responsibility for the contents of the article.
\section{Liouvillian first integrals and transversely affine foliations}
\label{sec:liouvillian-background}

In this section, we recall the basic theory of differential fields needed to define \emph{Liouvillian functions}. We do not aim to give a complete exposition, and for further details we refer to \cite{vanDerPutSinger2003,MR568864}.

\subsection{Liouvillian extensions}

\begin{definition}
\label{def:differential-field}
A \emph{differential field over $\C$} is a field $K\supset\C$ endowed with
a finite family
\[
   \Delta=\{\partial_1,\ldots,\partial_n\}
\]
of $K$-linearly independent, pairwise commuting $\C$-derivations. Its \emph{field of constants} is
\[
    K^\Delta=\left\{f\in K \mid \partial_i(f) = 0 \text{ for every } i \right\} = \bigcap_i\ker\partial_i \, .
\]
We write
\[
   \Omega^1_\Delta(K)=\Hom_K\left(\sum_iK\partial_i,K\right)
\]
and denote by $\dd_\Delta:K\to\Omega^1_\Delta(K)$ the differential $\dd_\Delta f=\sum_i\partial_i(f)\varepsilon_i$, where $\varepsilon_i(\partial_j)=\delta_{ij}$.
\end{definition}

\begin{definition}\label{def:differential-extension} 
Let $(K,\Delta)$ be a differential field. A \emph{differential extension} $L/K$ is a field extension equipped with commuting derivations extending those of $\Delta$. By abuse of notation, we also denote the extended derivations on $L$ by $\Delta$. Observe that this extension induces a canonical isomorphism
\[
    \Omega^1_\Delta(L) \simeq L\otimes_K\Omega^1_\Delta(K),
\]
under which we identify $\Omega^1_\Delta(K)$ with a subspace of $\Omega^1_\Delta(L)$.
\end{definition}

\begin{definition}
\label{def:liouvillian-extension}
Let $(K,\Delta)$ be a differential field with $K^\Delta=\C$. A differential extension $L/K$ is
\emph{Liouvillian} if $L^\Delta=\C$ and there is a tower of differential fields
\[
   K=K_0\subset K_1\subset\cdots\subset K_r=L,
   \qquad K_i=K_{i-1}(a_i),
\]
in which, for every $i$, either $a_i$ is algebraic over $K_{i-1}$, or
\[
   \dd_\Delta a_i\in\Omega^1_\Delta(K_{i-1}),
   \qquad\text{or}\qquad
   \frac{\dd_\Delta a_i}{a_i}\in\Omega^1_\Delta(K_{i-1}).
\]
The last two steps are called, respectively, adjunctions of a primitive and
of an exponential of a primitive.
\end{definition}

\begin{lemma}
\label{lem:liouvillian-correspondence}
Let $(K,\Delta)$ be a differential field with $K^{\Delta} = \C$, and let $K'/K$ be a finite algebraic extension. 
\begin{enumerate}[label=\textup{(\roman*)}]
\item If $L/K'$ is Liouvillian then $L/K$ is Liouvillian.
\item If $L/K$ is Liouvillian, then $LK'/K'$ is Liouvillian in any common differential overfield of $L$ and $K'$
\end{enumerate}
\end{lemma}
\begin{proof} 
We first remark that, since $K'/K$ is finite algebraic, the derivations in
$\Delta$ extend uniquely to commuting derivations on $K'$, making $K'/K$ a
differential field extension. Moreover, $(K')^{\Delta}=\C$. Indeed,
differentiating the minimal polynomial of $c\in(K')^{\Delta}$ shows that all
its coefficients are constant, and thus $c\in\C$.

For \textup{(i)}, the primitive element theorem gives $K'=K(a)$ with $a$
algebraic over $K$; adding this algebraic step at the beginning of a
Liouvillian tower for $L/K'$ yields a Liouvillian tower for $L/K$.

For \textup{(ii)}, take the compositum of every field in a tower for $L/K$
with $K'$. Each step remains an adjunction of the same type, since
$\Omega^1_\Delta(K_{i-1})\subset\Omega^1_\Delta(K_{i-1}K')$. Finally,
$(LK')^{\Delta}=\C$: the extension $LK'/L$ is algebraic and $L^{\Delta}=\C$,
so the argument of the first paragraph applies with $L$ in place of $K$.
\end{proof}

\subsection{Function fields}
We now specialize the definitions above to function fields. Let $X$ be a smooth variety of dimension $n$ and set $K=\C(X)$. Any transcendence basis $x_1,\ldots,x_n$ determines commuting derivations $\partial/\partial x_i$ on $K$, with common field of constants $\C$, and identifies
\[
   \Omega^1_{K/\C}=\Omega^1_\Delta(K)
\]
with the $K$-space of rational one-forms on $X$. Under this identification, $\dd_{\Delta}$ becomes the usual exterior derivative of rational functions, and hence we omit $\Delta$ from the notation, as well as in any differential field extension of $K$. We note that different choices of transcendence basis yield equivalent notions of differential extension and Liouvillianity. 

\begin{convention}
\label{conv:differential-function-field}
All differential fields used to define extensions of $K=\C(X)$ are embedded, compatibly with the chosen derivations, in the differential field $\mathcal M_{X,x}$ of germs of meromorphic functions at a general point. 
\end{convention}

\begin{remark}\label{rmk:constants-automatic}
    Two consequences of \Cref{conv:differential-function-field} will be used without further comment. First, any subfield of $\mathcal M_{X,x}$ containing $\C$ and stable under the derivations has field of constants $\C$, since a germ killed by all derivations is constant. Hence the condition $L^\Delta=\C$ in \Cref{def:liouvillian-extension} is automatic for every extension considered in this paper. Second, the compositum of two Liouvillian
extensions of $K$ inside $\mathcal M_{X,x}$ is again Liouvillian:
concatenate a tower for the first with the compositum of a tower for the second, as in the proof of \Cref{lem:liouvillian-correspondence}.
\end{remark}

\begin{definition}\label{def:defined-over}
    Let $L/K$ be a differential field extension. A (germ of) function is \emph{defined over} $L$ if it belongs to $L\subset \mathcal M_{X,x}$. Similarly, a (germ of) one-form is defined over $L$ if it belongs to $L\otimes_K\Omega^1_{K/\C}$. We say the respective function or one-form is \emph{Liouvillian} if it is defined over a Liouvillian extension of $K$.
\end{definition}

\begin{remark}\label{rmk:liouvillian-function-form}
    A function $f$ is Liouvillian if and only if $\dd f$ is Liouvillian.
\end{remark}

\subsection{Liouvillian first integrals}
We refer to \cite{CousinPereira2014} for the basic notions and terminology concerning foliations used throughout this text.

\begin{definition}
\label{def:global-liouvillian-fi}
Let $\F$ be a codimension-one foliation on a smooth projective variety $X$.
A \emph{global Liouvillian first integral} is a nonconstant $a$ in a
Liouvillian extension of $\C(X)$ such that every rational vector field
tangent to $\F$ annihilates $a$.
\end{definition}

Equivalently, if $\omega$ is a rational one-form defining $\F$, then
$\dd a\wedge\omega=0$.

\begin{definition}
\label{def:ta}
A \emph{singular transverse affine structure} on a codimension-one
foliation $\F$ is a flat meromorphic connection on $N_\F$ whose restriction
to $T_\F$ is the Bott partial connection. Equivalently, for a rational
one-form $\omega$ defining $\F$, a transverse affine structure is represented by a rational one-form $\eta$ satisfying
\begin{equation}
\label{eq:transverse-affine-structure}
   \dd\omega=\omega\wedge\eta,
   \qquad
   \dd\eta=0.
\end{equation}
Under a rational rescaling $\omega\mapsto g\omega$, the corresponding $\eta$ changes to
$\eta-\dd g/g$.  A foliation carrying such a structure is
\emph{transversely affine}. Similarly, $\F$ is called \emph{transversely additive} if it can be defined by a closed rational one-form. Equivalently, it admits a singular transverse affine structure with $\eta =0$.
\end{definition}

See \cite[Section~2.2]{CousinPereira2014} for the equivalence between the
connection and one-form descriptions of transverse affine structures.

\begin{theorem}[Singer's criterion]
\label{thm:singer-criterion}
A codimension-one foliation on a smooth complex projective variety admits a
global Liouvillian first integral if and only if it is transversely affine.
\end{theorem}

For rational surfaces this is Singer's theorem \cite[Theorem~1]{Singer1992}; the
general case follows similarly and is spelled out in \cite[Theorem~5.4]{Scardua2011}, see also \cite[Section 2]{CousinPereira2014}. The easy
direction is immediate from \eqref{eq:transverse-affine-structure}: if $\dd E/E=\eta$, then $E\omega$ is closed, and a primitive of $E\omega$ is a Liouvillian first integral. In this paper, only the elementary implication, transversely affine implies Liouvillian first integral, is used.
\section{Webs and abelian relations}
\label{sec:webs-abelian-relations}

Throughout the paper, $X$ denotes a smooth projective variety of dimension
$n\geq2$, and all foliations and webs have codimension one.
We follow the conventions of \cite{PereiraPirio2015}, except that the
web germs called smooth here are only required to be pairwise transverse;
they are called quasi-smooth in that reference.

\subsection{Webs}

\begin{definition}
\label{def:germ-web}
A \emph{germ of a smooth $k$-web} at $x\in X$ is a
superposition
\[
   \W_x=\F_{1,x}\boxt\cdots\boxt\F_{k,x}
\]
of germs of foliations with pairwise distinct tangent
hyperplanes at $x$. If $\omega_i$ defines $\F_{i,x}$, then
$(\omega_i\wedge\omega_j)(x)\neq0$ for $i\neq j$, and
$\W_x$ is represented by the symmetric product
$\omega_1\cdots\omega_k\in\Sym^k\Omega^1_{X,x}$.
\end{definition}

The superposition $\W_x\boxt\W'_x$ of two smooth web germs
is defined in the same way, provided all their tangent
hyperplanes at $x$ are distinct.

\begin{definition}
\label{def:global-web}
A $k$-web on $X$ is a saturated line subsheaf
\[
   N^*_{\W}\subset\Sym^k\Omega_X^1
\]
whose local generator near a general point factors as
\[
   \omega_1\cdots\omega_k,
   \qquad \omega_i\wedge\dd\omega_i=0,
   \qquad \omega_i\wedge\omega_j\neq0\quad(i\neq j).
\]
Its discriminant $\Delta(\W)$ is the reduced locus where
it is not a smooth web. We set
$U_{\W}=X\setminus\Delta(\W)$.
On $U_{\W}$, the web splits locally into smooth foliations,
uniquely up to permutation.
\end{definition}

A foliation $\F$ is \emph{transversal} to $\W$ if its
generic conormal direction is not a factor of $\W$.
Their product then defines a $(k+1)$-web $\W\boxt\F$.

The local factors of a web need not define foliations on
$X$, but they do so after passing to a suitable cover
\cite[Sections~2.2 and~3.1]{FavrePereira2015}.

\begin{definition}
\label{dfn:pullback-web}
For a generically finite morphism $\pi:X'\to X$ of smooth
projective varieties, the pullback $\pi^*\W$ is obtained by
pulling back $N^*_{\W}\to\Sym^k\Omega_X^1$ and saturating
its image in $\Sym^k\Omega_{X'}^1$.
\end{definition}

\begin{proposition}[Decomposing cover]
\label{prop:decomposable-pullback}
Let $\W$ be a $k$-web on $X$. There are a finite Galois
extension $K'/\C(X)$, a smooth projective variety $X'$ with
$\C(X')=K'$, and a generically finite morphism $\pi:X'\to X$
such that
\[
   \pi^*\W=\F'_1\boxt\cdots\boxt\F'_k.
\]
The Galois group permutes the factors, and $\pi$ is finite
\'etale and Galois over a dense open subset of $U_{\W}$.
\end{proposition}

\begin{proof}
At the generic point, $\W$ is represented by a symmetric form
which is geometrically a product of distinct linear factors.
Choose a finite Galois extension $K'/\C(X)$ over which all these
factors are defined. By integrability, they define foliations over $K'$,
permuted by the Galois group.
Normalize $X$ in $K'$ and take an
equivariant resolution $X'$. The resulting morphism
$\pi:X'\to X$ has the required decomposition. Away from
the branch locus and the image of the exceptional locus,
it is finite \'etale and Galois.
\end{proof}

\begin{convention}
\label{conv:decomposed-web}
Whenever an ordered list of factors is needed, we pass
to a decomposing cover, choose an ordering, and suppress
the primes. Thus $\W=\F_1\boxt\cdots\boxt\F_k$ may denote
either a local splitting or the rational splitting on
that cover.
\end{convention}

\subsection{Abelian relations}

\begin{definition}
\label{def:abelian-relations}
Let $x\in U_{\W}$ and choose a local ordering
$\W=\F_1\boxt\cdots\boxt\F_k$.
An \emph{abelian relation} at $x$ is a tuple
$(\eta_1,\ldots,\eta_k)$ of germs of closed one-forms
such that $\eta_i\in N_{\F_i}^*$ and
\[
   \sum_{i=1}^k\eta_i=0.
\]
The vector spaces $A(\W)_x$ of abelian relations form a
local system $\Acal_{\W}$ on $U_{\W}$, of rank at most
$(k-1)(k-2)/2$ \cite[Chapter~2]{PereiraPirio2015}.
We denote its rank by $\rank(\W)$. 

Let $\OO_{X/\F_i}$ be the ring of germs of holomorphic
first integrals of $\F_i$. A \emph{functional abelian
relation} at $x$ is a tuple $(u_1,\ldots,u_k)$ with $u_i\in\OO_{X/\F_i,x}$ and 
\[
    \sum_{i=1}^k u_i=0.
\]
These relations form a vector space $\Sol(\W)_x$,
also the stalk of a local system on $U_{\W}$.
\end{definition}

The Poincar\'e lemma gives the exact sequence
\begin{equation}
\label{eq:functional-abelian-exact}
0\longrightarrow
\left\{(c_i)\in\C^k\mathrel{\Big|}\sum_i c_i=0\right\}
\longrightarrow\Sol(\W)_x
\xrightarrow{\,\dd\,}A(\W)_x
\longrightarrow0.
\end{equation}

\begin{definition}
\label{def:component-space}
Let $\F$ be a foliation transversal to $\W$ and let $x$
be a regular point of $\W\boxt\F$. We denote by
$\Sol_{\F}(\W)_x$ the image of the projection
\[
   \Sol(\W\boxt\F)_x\longrightarrow\OO_{X/\F,x}
\]
onto the $\F$-component, and define $A_{\F}(\W)_x$
similarly. Differentiation gives an exact sequence
\begin{equation}
\label{eq:component-exact}
   0\longrightarrow\C\longrightarrow\Sol_{\F}(\W)_x
   \xrightarrow{\,\dd\,}A_{\F}(\W)_x
   \longrightarrow0.
\end{equation}
\end{definition}

\begin{definition}
\label{def:liouvillian-abelian-relation}
An abelian relation is \emph{defined over} a differential
extension $L/\C(X)$ if all its components are defined over
$L$ in the sense of \Cref{def:defined-over}.
It is \emph{Liouvillian} if $L$ can be chosen Liouvillian.
The same terminology applies to functional abelian relations.
\end{definition}

By \Cref{rmk:liouvillian-function-form}, an abelian relation
is Liouvillian if and only if any of its functional lifts
in \eqref{eq:functional-abelian-exact} is Liouvillian.

\begin{remark}
\label{rmk:liouvillian-descent-covers}
Let $\pi:X'\to X$ be a generically finite morphism of smooth
projective varieties, and set $K=\C(X)$ and $K'=\C(X')$.
Choose $x\in U_{\W}$ and $x'\in\pi^{-1}(x)$ where $\pi$
is \'etale, and identify the meromorphic germs at $x$
and $x'$ through $\pi$. Then an abelian relation $\alpha$
at $x$ is Liouvillian if and only if $\pi^*\alpha$ is
Liouvillian. Indeed, if $\alpha$ is defined over a
Liouvillian extension $L/K$, then $\pi^*\alpha$ is defined
over $LK'/K'$, which is Liouvillian by
\Cref{lem:liouvillian-correspondence}\textup{(ii)}.
Conversely, a Liouvillian extension of $K'$ is also
Liouvillian over $K$ by
\Cref{lem:liouvillian-correspondence}\textup{(i)}.
The same argument applies to functional abelian relations.
\end{remark}
\section{Rational transverse differential equations}\label{sec:rational-transverse-equation}

Let $X$ be a smooth projective variety, let $\W$ be a $k$-web, and let $\F$ be a foliation transversal to $\W$. Whenever necessary, after passing to a decomposing cover, we suppose 
\[
   \W\boxt\F=\F_1\boxt\cdots\boxt\F_k\boxt\F.
\]

In this section, we construct a rational transverse differential equation satisfied by every element of $\Sol_{\F}(\W)$. We use the transverse jet formalism of \cite{Fazoli2025a,Fazoli2025b}. For the basic theory of jets and differential equations, we refer to \cite{Deligne1970,Spencer69}.

\subsection{Transverse jets}

For $N\geq0$, let $\jet^N$ be the bundle of $N$-jets of functions, and let
\[
j^N: \OO_X \rightarrow \jet^N
\]
be the $\C$-linear morphism corresponding to taking the $N$-jet of a function $f$.

Let $\F$ be a foliation on $X$, set $K=\C(X)$, and choose a $K$-basis
$v_1,\ldots,v_{n-1}$ of $T_{\F}\otimes K$. The germs of first integrals
of $\F$ are exactly the common solutions of
\[
   v_a(f)=0\qquad(1\leq a\leq n-1).
\]

For a rational vector field $v$ tangent to $\F$ and $N\geq1$, the
\emph{$(N-1)$-th prolongation} of $v$ (see
\cite[Definition~1.2.3]{Spencer69}) is the $\OO_X$-linear morphism
\[
   v^{(N-1)}:\jet^N\longrightarrow\jet^{N-1}\otimes\mathcal M_X
\]
determined by
\[
   v^{(N-1)}\bigl(j^N f\bigr)=j^{N-1}\bigl(v(f)\bigr)
\]
for every germ of an analytic function $f$. Its kernel consists of the
$N$-jets that solve $v(f)=0$ up to order $N$.

\begin{definition}\label{def:transverse jets}
We define $T^N\subset\jet^N$ by
\[
   T^N=\bigcap_{a=1}^{n-1}\ker v_a^{(N-1)}.
\]
Following the
terminology of \cite[\S~3.8]{Fazoli2025b}, we refer to this sheaf as the
sheaf of \emph{transverse jets}
\[
   T^N=\mathcal P^N_{X/\F}(\dd_{\F}),
\]
where $\dd_{\F}:\OO_X\rightarrow\Omega^1_{\F}$ is the derivation along the
leaves of $\F$.
\end{definition}

The common kernel is unchanged under an invertible rational change of
tangent frame, so $T^N$ depends only on $\F$. In foliated coordinates
$(x_1,\ldots,x_{n-1},t)$ at a regular point, with $\F$ defined by $\dd t$,
use the local tangent frame $\partial_{x_1},\ldots,\partial_{x_{n-1}}$.
The conditions $j^{N-1}(\partial_{x_a}f)=0$ annihilate exactly the jet
coefficients involving at least one $x_a$. Hence, over the regular locus
of $\F$, $T^N$ is a subbundle of $\jet^N$ of rank $N+1$, whose fibre at
each point is generated by the $N$-jets of the local first integrals of $\F$.

The bundle $T^N$ carries a \emph{canonical flat $\F$-partial connection}, denoted by $\partial^N$. Over the regular locus of $\F$ it can be described directly. For a germ of an analytic function $f$, we already noticed that
\[
j^N f \in T^N \iff f\in \OO_{X/\F}.
\]
The connection $\partial^N$ is the unique flat partial connection on $T^N$ such that
\[
\partial^N(j^N f) =0 \quad\text{for every } f\in \OO_{X/\F};
\]
in the frame of a foliated coordinate introduced below, it is the derivative of the coefficients along the leaves. This description is all that will be used in this paper. The intrinsic construction of $T^N$ and $\partial^N$ on the whole of $X$, including the singular set of $\F$, is carried out in \cite[\S~3.8]{Fazoli2025b}, see also \cite[\S~4.3]{Fazoli2025a}.

For the purpose of this work, we will need suitable frames for $T^N$, which
can be obtained using \emph{infinitesimal automorphisms}. Recall that an
infinitesimal automorphism of $\F$ is a vector field $w$ such that
\[
    [w,v] \in T_{\F} \text{ for every } v\in T_{\F}.
\]
Equivalently, denoting also by $w$ the corresponding section of $N_{\F}$, we
have
\[
    \nabla_{\Bott} w=0,
\]
where $\nabla_{\Bott}$ is the Bott connection. An infinitesimal automorphism
tangent to $\F$ induces the zero section of $N_{\F}$; from now on, we only
consider infinitesimal automorphisms that are \emph{generically transverse}
to $\F$, that is, that induce a nonzero flat section of $N_{\F}$.

Every infinitesimal automorphism preserves local first integrals: if
$f\in\OO_{X/\F}$ and $v$ is tangent to $\F$, then
\[
   v\bigl(w(f)\bigr)=w\bigl(v(f)\bigr)+[v,w](f)=0,
\]
and thus $w(f)\in\OO_{X/\F}$. Over the open set where $w$ is regular and
transverse to $\F$, it determines a frame of $T^N$ in which
\[
j^N f = \bigl(f, w(f), \ldots, w^N(f)\bigr) \text{ for every } f\in \OO_{X/\F}.
\]
In particular, a foliated coordinate $t$ determines the infinitesimal
automorphism $\partial_t$, and the natural frame determined by it is such
that
\[
j^N (G\circ t) = \bigl(G(t), G'(t), \ldots, G^{(N)}(t)\bigr).
\]

Finally, we conclude our discussion of transverse jets by describing a naturally associated short exact sequence. For $N\geq1$, regarding $T^N$ as a rational bundle on $X$, truncation of the jets yields an exact sequence of flat partial connections
\begin{equation}
\label{eq:transverse-jet-sequence}
0\longrightarrow
\bigl((\CNF)^{\otimes N},\nabla_{\Bott}^{\otimes N}\bigr)
\longrightarrow(T^N,\partial^N)
\longrightarrow(T^{N-1},\partial^{N-1})\longrightarrow 0,
\end{equation}
where $\nabla_{\Bott}$ is the Bott connection. The exact sequence is compatible with the frame description of the bundles. In a foliated coordinate $t$, the symbol line is generated by $\dd t^{\otimes N}$.

\subsection{Formal relations and rational annihilators}
\label{subsec:rationality-transversality}

Following \Cref{def:transverse jets}, let $T_i^N\subset \jet^N$ be the respective transverse jets associated to the foliations $\F_i$. We set $R^N$ to be the bundle of \emph{formal functional relations} on $\jet^N$, that is,
\[
R^N=\left\{(\xi_1,\ldots,\xi_k,\xi)\ \middle|\
 \xi_i\in T_i^N,\ \xi\in T^N,\ \xi + \sum_{i=1}^k\xi_i=0\text{ in }\jet^N\right\},
\]
and let $R^N_{\F}\subset T^N$ be its image under projection to the last factor. Remark that $j^N u \in R^N_{\F}$ for every $u\in \Sol_{\F}(\W)$. 

\begin{remark}
Recall that the distinguished foliation $\F$ is defined on the original
variety $X$. Set $K=\C(X)$ and choose a Galois decomposing cover with
function field $K'$. At the generic point, write
$T'^N=K'\otimes_K T^N$ and let $T_i'^N$ be the transverse jet spaces
of the factors of the pulled-back web. The projected space on the cover is
\[
    R_{\F}'{}^N=T'^N\cap\bigl(T_1'^N+\cdots+T_k'^N\bigr).
\]
The Galois group preserves $T'^N$ and permutes the $T_i'^N$, so
$R_{\F}'{}^N$ is Galois invariant. Hence this space and its annihilator
descend to $K$. We therefore regard $R_{\F}^N$ as a rational subbundle
of $T^N$ on $X$ and take all rational annihilators below over $\C(X)$.     
\end{remark}

Recall that an element $\Lambda \in (T^N)^\vee$ is a \emph{transverse
differential operator}, as in \cite[\S~2]{Fazoli2025b}. Roughly speaking, it
corresponds to the restriction of a differential operator to the ring of
first integrals of $\F$. We consider the rational transverse differential
operators annihilating the general fibres of $R^N_{\F}$, that is,
\[
\Ann(R^N_{\F}) = \left\{\Lambda \in (T^N)^\vee\otimes\C(X) \;\middle|\;
\Lambda\bigl((R^N_{\F})_x\bigr)=0 \text{ for general } x\in X \right\}.
\]
We also consider the $\C(X)$-space of rational sections of $(T^N)^\vee$
vanishing on the $N$-jets of elements of $\Sol_{\F}(\W)$, that is,
\[
    \mathscr I^N_{\F}(\W) = \left\{\Lambda \in (T^N)^\vee\otimes \C(X)
    \;\middle|\; \Lambda(j^N u)=0 \text{ for every } u\in \Sol_{\F}(\W)
    \right\}.
\]
These spaces satisfy
\begin{equation}\label{eq:inclusion-rational-annihilator}
        \Ann(R^N_{\F}) \subset \mathscr I^N_{\F}(\W).
\end{equation}
Indeed, let $\Lambda\in\Ann(R^N_{\F})$ and $u\in\Sol_{\F}(\W)$. At every
point $y$ close to the base point, the germ of $u$ at $y$ still belongs to
$\Sol_{\F}(\W)_y$, and hence $j^Nu\in(R^N_{\F})_y$. Therefore the
meromorphic germ $\Lambda(j^Nu)$ vanishes at every such general $y$, and thus
it vanishes identically.

\begin{lemma}
\label{lem:rational-existence}
Set $N_0=k(k+1)/2$. For $N\geq N_0$, $\mathscr I^N_{\F}(\W)\neq0$.
\end{lemma}

\begin{proof}
    By \eqref{eq:inclusion-rational-annihilator}, it is enough to prove that
    $\Ann(R^N_{\F})\neq 0$. For $M\geq1$, an element of the kernel of
    $R^M \rightarrow R^{M-1}$ is a tuple of pure symbols
    $\big( c_1\omega_1^{\otimes M},\ldots, c_k\omega_k^{\otimes M},
    c\,\omega^{\otimes M}\big)$ satisfying
    \[
        c_1\omega_1^{\otimes M}
        +\cdots+
        c_k\omega_k^{\otimes M}
        +c\,\omega^{\otimes M}
        =0.
    \]
At a point $x\in U_{\W\boxt\F}$, choose a general two-dimensional
linear subspace $P\subset T_xX$ such that the restrictions of the $k+1$
conormals to $P$ are nonzero and pairwise nonproportional. Their $M$-th
powers determine distinct points of the rational normal curve in
$\Pj(\Sym^M P^*)$, and span a vector space of dimension
$\min(k+1,M+1)$. Therefore
\[
   \dim\operatorname{Span}
   \{\omega_1^{\otimes M},\ldots,\omega_k^{\otimes M},
     \omega^{\otimes M}\}
   \geq\min(k+1,M+1),
\]
and hence
\[
   \dim\ker(R^M_x\longrightarrow R^{M-1}_x)
   \leq\max(0,k-M).
\]
    Since $\rank R^0_x=k$,
    \[
        \rank R^N_x\leq k+\sum_{M=1}^{k-1}(k-M)=\frac{k(k+1)}2=N_0.
    \]
    Thus, for $N\geq N_0$ and every $x\in U_{\W\boxt\F}$,
    \[
        \rank (R^N_{\F})_x\leq N_0<\rank T^N=N+1.
    \]
    Since the dimension of the fibres $(R^N_{\F})_x$ is constant on a dense open subset, linear algebra over $\C(X)$ produces a nonzero rational section of $(T^N)^\vee$ vanishing on the general fibres of $R^N_{\F}$. Therefore $\Ann(R^N_{\F})\neq0$.
\end{proof}

\subsection{Minimal rational annihilators}

\begin{definition}
\label{def:rational-relation-module} A \emph{rational annihilator of order at most $N$} is an element of $\mathscr I^N_{\F}(\W)$.  If $\Lambda\neq 0$, its \emph{order} is the smallest integer $m\leq N$ such that $\Lambda$ factors through the truncation $T^N\to T^m$. The \emph{symbol} of $\Lambda$ is its restriction to $(\CNF)^{\otimes m}$ using \eqref{eq:transverse-jet-sequence}, defining
\[
   \sigma(\Lambda):(\CNF)^{\otimes m}\rightarrow \mathcal M_X.
\]
\end{definition}

Equivalently, let $w$ be an infinitesimal automorphism of $\F$ and $\Lambda \in \mathscr I^N_{\F}(\W)$ be a rational annihilator. We write 
\begin{equation}\label{eq:transverse-diff-operator}
    \Lambda \big(j^N f \big) = \sum_{\ell =0}^N c_{\ell} \, w^{\ell}(f),
\end{equation}
for all $f\in \OO_{X/\F}$. Since constant functions are always solutions of $\Lambda$, $c_0=0$. The order of $\Lambda$ is the largest $\ell$ for which $c_{\ell} \neq 0$, and its symbol is $c_m \, w^{m}$. In the particular case where $w$ is induced by $\partial_t$ for a foliated coordinate $t$, this becomes
\[
   \Lambda\big(j^N(G\circ t)\big)  = \sum_{\ell=1}^m c_\ell \, \big(G^{(\ell)}\circ t\big).
\]
The dual bundle $(T^N)^\vee$ carries the dual flat partial connection, given for $v\in T_{\F}$ by
\[
   \Bigl(\bigl(\partial^N\bigr)^\vee_v\Lambda\Bigr)(s)=v\bigl(\Lambda(s)\bigr)-\Lambda\bigl(\partial^N_v s\bigr);
\]
in terms of differential operators, it is induced by the commutator $[v,\Lambda]$, see \cite[Theorem~2.4]{Fazoli2025b}. If $w$ is an infinitesimal automorphism and we express $\Lambda$ as in Equation \eqref{eq:transverse-diff-operator}, then, evaluating on the horizontal sections $j^Nf$ with $f\in\OO_{X/\F}$ and using that each $w^{\ell}(f)$ is again a first integral,
\begin{equation}\label{eq:connection-transverse-diff-operator}
    \Big(\big(\partial^N\big)^\vee_v \Lambda \Big)( j^N f) = \sum_{\ell =1}^N v(c_{\ell}) \, w^{\ell}(f).
\end{equation}

\begin{lemma}
\label{lem:rational-transversality}
Assume $\Sol_{\F}(\W)\neq\C$. Let $m$ be the least order of a nonzero
rational annihilator and let $\Lambda$ have order $m$. Then
\[
   2\leq m\leq N_0,
\]
and the following hold.
\begin{enumerate}[label=\textup{(\roman*)}]
\item The annihilators of order $m$ form the rational line $\C(X)\Lambda \subset \mathscr I^m_{\F}(\W)$.
\item This rational line is preserved by the canonical connection on $(T^N)^\vee$.

\item If $w$ is an infinitesimal automorphism of $\F$, generically transverse to $\F$ and defined over a Liouvillian extension $L$, the normalized differential operator induced by $\Lambda$ is
\begin{equation*}
   \mathcal D_w=
   w^m+\sum_{\ell=1}^{m-1}c_\ell\, w^\ell,
\end{equation*}
where $c_{\ell}$ are first integrals of $\F$ over $L$. Moreover, $\mathcal D_w$ does not depend on $\Lambda$.

\item In a local first integral $t$ of $\F$, the normalized equation is
\begin{equation*}
   \mathcal D_t=
   \partial_t^m+\sum_{\ell=1}^{m-1}A_\ell(t)\partial_t^\ell,
\end{equation*}
where the $A_\ell$ are meromorphic first integrals of $\F$. Moreover, $\mathcal D_t$ does not depend on $\Lambda$. If $t=h(\bar t)$ is a change of foliated coordinates, then
\[
   \mathcal D_{\bar t}(G\circ h)
   =(h'\circ\bar t)^m(\mathcal D_tG)\circ h.
\]
\end{enumerate}
\end{lemma}

\begin{proof}
The existence of a minimal rational annihilator and the upper bound follow from \Cref{lem:rational-existence}. For the lower bound, observe that a differential operator of order at most one admitting all constant functions as solutions and also a non-constant solution must be identically zero. 

To prove \textup{(i)}, let $\Lambda'$ be another rational annihilator of minimal order. Since the symbols $\sigma(\Lambda)$ and $\sigma(\Lambda')$ are rational sections of the same line-bundle $N_{\F}^{\otimes m}$, there is a non-zero rational function $\rho$ such that 
\[
\sigma(\Lambda) - \rho \, \sigma(\Lambda')=0.
\]
Hence the symbol of $\Lambda- \rho \, \Lambda'$ is zero, and thus it has order at most $m-1$.  Since $m$ is the minimal order of a non-vanishing rational annihilator, it follows that 
\[
\Lambda = \rho \, \Lambda'.
\]
This proves \textup{(i)}.

For every rational vector field $v$ tangent to $\F$ and $f\in \Sol_{\F}(\W)$, we have that
\[
\Big(\big(\partial^N \big)^\vee_v \Lambda\Big)(j^m f) = v(\Lambda(j^m f)) - \Lambda(\partial^N_v j^m f) =0 ,
\]
thus $\big(\partial^N \big)^\vee_v \Lambda \in \mathscr I_{\F}^m(\W)$. Since $\big(\partial^N \big)^\vee_v \Lambda$ also has order $\le m$, claim \textup{(ii)} follows from \textup{(i)}.

For \textup{(iii)}, since $w$ is generically transverse to $\F$, the map
$j^N f\mapsto(f,w(f),\ldots,w^N(f))$ trivializes $T^N$ over $L$, so that
$\Lambda$ can be written as in \eqref{eq:transverse-diff-operator} with
coefficients $b_\ell\in L$ and $b_m\neq0$; set
$D_w=b_m^{-1}\Lambda$ and $c_\ell=b_\ell/b_m$ for $\ell<m$.
Extend the canonical partial connection to $L$ by the Leibniz rule.
For every rational vector field $v$ tangent to $\F$, \textup{(ii)} gives
$\lambda(v)\in\C(X)$ such that
$\bigl(\partial^N\bigr)^\vee_v\Lambda=\lambda(v)\Lambda$, and the Leibniz rule
gives
\[
   \bigl(\partial^N\bigr)^\vee_v(D_w)
   =\Bigl(v\bigl(b_m^{-1}\bigr)+b_m^{-1}\lambda(v)\Bigr)\,\Lambda.
\]
By Equation~\eqref{eq:connection-transverse-diff-operator}, the left-hand
side has order at most $m-1$, since the coefficient of $w^m$ in $D_w$ is
constant, while a nonzero multiple of $\Lambda$ has order exactly $m$. Hence
both sides vanish. Applying
Equation~\eqref{eq:connection-transverse-diff-operator} once more, we obtain
$\sum_{\ell=1}^{m-1} v(c_\ell)\,w^\ell(f)=0$ for every local first integral
$f$, and therefore $v(c_\ell)=0$ for every $\ell$ and every
rational vector field $v$ tangent to $\F$; that is, each $c_\ell$ is a first
integral of $\F$ in $L$. Finally, $D_w$ does not depend on $\Lambda$, since
by \textup{(i)} any other choice is $\rho\Lambda$ with $\rho\in\C(X)^*$, and
normalization removes $\rho$.

The first part of \textup{(iv)} is analogous to the proof of \textup{(iii)}. The second claim follows from the chain rule, since by \textup{(i)} it is enough to compute the transformation law of the symbol.
\end{proof}

\subsection{The normalized transverse equation}
\label{subsec:abel-fpc}
Let $\Lambda$ be a minimal rational annihilator of order $m$. Its symbol
is a generically invertible rational morphism
\[
   \sigma(\Lambda):
   (\CNF)^{\otimes m}\rightarrow\mathcal M_X.
\]
Define
\begin{equation}\label{eq:normalized-transverse-equation}
   D_\Lambda
   =
   \sigma(\Lambda)^{-1}\circ\Lambda:
   T^m\rightarrow \mathcal M_X \otimes (\CNF)^{\otimes m}.
\end{equation}
The morphism $D_\Lambda$ is independent of the multiplication of
$\Lambda$ by an element of $\C(X)^*$, and it defines a rational splitting of
the short exact sequence
\[
   0\longrightarrow
   \big((\CNF)^{\otimes m},\nabla_{\Bott}^{\otimes m}\big)
   \longrightarrow
   (T^m,\partial^m)
   \longrightarrow
   (T^{m-1},\partial^{m-1})
   \longrightarrow 0.
\]
Moreover, this splitting is horizontal. Indeed, in the frame determined by a
foliated coordinate $t$,
\[
 D_\Lambda(j^m(G\circ t))=
 \left(G^{(m)}(t)+\sum_{\ell=1}^{m-1}A_\ell(t) \, G^{(\ell)}(t)\right)
 \dd t^{\otimes m},
\]
and, by \textup{(iv)} of \Cref{lem:rational-transversality}, the coefficients
$A_\ell$ are meromorphic first integrals of $\F$; this is precisely the
compatibility of $D_\Lambda$ with the partial connections $\partial^m$ and
$\nabla_{\Bott}^{\otimes m}$. Hence $D_\Lambda$ is a rational transverse
differential equation of order $m\geq2$ on $(\OO_X,\dd_{\F})$ in the sense of
\cite{Fazoli2025a}. This is the transverse differential equation that will be used in the
proof of \Cref{thm:abel-to-affine}.

We remark that, when $\F$ admits a rational infinitesimal automorphism $w$,
in the frame determined by it, one has
\[
   D_\Lambda\big(j^m f\big)
   =
   \left(
      w^m(f)
      +
      \sum_{\ell=1}^{m-1}
         c_{\ell}\, w^{\ell}(f)
   \right)\omega^{\otimes m},
\]
where $\omega$ is the rational conormal one-form determined by $\omega(w)=1$
and, by \textup{(iii)} of \Cref{lem:rational-transversality}, each $c_{\ell}$
is a rational first integral of $\F$. We will also use this special case
below.

\begin{remark}
It is useful to compare the rational transverse differential equation constructed above with other natural differential equations associated with a web. See, for instance, the remark following \cite[Proposition~2.1.3]{Pirio2004}.

The first is the \emph{Wronskian} associated with $\Sol_{\F}(\W)$, namely the
minimal-order analytic differential equation whose solution space is
$\Sol_{\F}(\W)$. Its order equals $\dim_{\C}\Sol_{\F}(\W)_x$ and never
exceeds the order of $D_{\Lambda}$: indeed, $\Sol_{\F}(\W)_x$ is contained in
the space of local solutions of $D_{\Lambda}$, which is a $\C$-vector space
of dimension at most $m$.  The Wronskian construction, however, does not by itself ensure
that the coefficients are rational, and therefore does not
directly establish Liouvillianity over $\C(X)$.

A second construction is provided by the \emph{Abel method} (see \cite{Pirio2006Selecta}). Although the differential equations obtained by this method are initially defined only locally, one can verify that, for webs without monodromy, the resulting local equations are compatible on intersections. Consequently, they define a global transverse differential equation. The order of this equation may be considerably larger than the order of $D_{\Lambda}$, but the Abel method has the advantage of being constructive. We will not make use of this construction here.

A third construction is H\'enaut's connection
\cite{Henaut2004}: for a planar $k$-web, the abelian relations are identified
with the flat sections of a rational connection on a bundle of rank
$(k-1)(k-2)/2$; see also \cite[\S~6.3]{PereiraPirio2015}. That connection
governs all the components of the abelian relations simultaneously, while
$D_\Lambda$ isolates a single transverse direction. It is this
one-direction-at-a-time nature, together with the control of the coefficients
provided by \Cref{lem:rational-transversality}, that the arguments of
\Cref{sec:no-rational-first-integrals,sec:component-modules,sec:resonance-descent,sec:proof-main-theorem} require.
\end{remark}

\section{From local abelian relations to a transverse affine structure}\label{sec:transverse-affine-structure}

We now extract a transverse affine structure from the normalized transverse
differential equation constructed in \eqref{eq:normalized-transverse-equation}. 

\begin{theorem}\label{thm:abel-to-affine}
Let $X$ be a smooth projective variety, let $\W$ be a $k$-web, and let $\F$
be a foliation transversal to $\W$. If $\Sol_{\F}(\W)_x\neq\C$ at a general regular point $x$, then $\F$ is transversely affine.
\end{theorem}

\begin{proof}
Let $\Lambda$ be a rational annihilator of minimal order $m\geq2$, as in
\Cref{lem:rational-transversality}, and let $D_\Lambda$ be its normalized
transverse equation \eqref{eq:normalized-transverse-equation}. Let $u$ be a
local first integral of $\F$ on an open set $V$ contained in the regular
locus of $\F$. By \textup{(iv)} of \Cref{lem:rational-transversality},
\[
   D_\Lambda\bigl(j^m(G\circ u)\bigr)
   =\Bigl(G^{(m)}(u)+\sum_{\ell=1}^{m-1}A_{\ell,u}(u)\,G^{(\ell)}(u)\Bigr)
   \dd u^{\otimes m},
\]
where the $A_{\ell,u}$ are meromorphic first integrals on $V$; there is no
zero-order term because the constants belong to $\Sol_{\F}(\W)_x$. Set
$q=\binom{m}{2}$, which is positive since $m\geq2$, and
\[
   \theta_u=\frac1q\,A_{m-1,u}(u)\,\dd u,
\]
a closed meromorphic one-form on $V$ vanishing on $T_{\F}$.

We first determine how $\theta_u$ depends on $u$. Let $\bar u$ be another
local first integral on $V$, so that $u=h(\bar u)$ for a holomorphic
function $h$ with $a=h'\circ\bar u$ nowhere zero. Writing
$D_{\Lambda,u}=\partial_u^{\,m}+\sum_{\ell}A_{\ell,u}\,\partial_u^{\,\ell}$,
and similarly for $\bar u$, \textup{(iv)} of \Cref{lem:rational-transversality}
gives $D_{\Lambda,\bar u}(G\circ h)=a^m\,(D_{\Lambda,u}G)\circ h$. Since
$G^{(\ell)}\circ h=(a^{-1}\partial_{\bar u})^{\ell}(G\circ h)$ and
\[
   a^m\bigl(a^{-1}\partial_{\bar u}\bigr)^m
   =\partial_{\bar u}^{\,m}
   -q\,\frac{\partial_{\bar u}a}{a}\,\partial_{\bar u}^{\,m-1}
   +\text{terms of order at most $m-2$},
\]
comparing the coefficients of $\partial_{\bar u}^{\,m-1}$ yields
\[
   A_{m-1,\bar u}=a\,(A_{m-1,u}\circ h)-q\,\frac{\partial_{\bar u}a}{a},
   \qquad\text{that is,}\qquad
   \theta_{\bar u}=\theta_u-\dd\log a .
\]

Now define a connection $\nabla_{\mathrm{aff}}$ on $\CNF$ over the regular
locus of $\F$ by requiring
\[
   \nabla_{\mathrm{aff}}(\dd u)=\dd u\otimes\theta_u
\]
for every local first integral $u$. This is consistent: if $u=h(\bar u)$ as
above, then $\dd u=a\,\dd\bar u$ and
\[
   \nabla_{\mathrm{aff}}(a\,\dd\bar u)
   =\dd a\otimes\dd\bar u+a\,\dd\bar u\otimes\theta_{\bar u}
   =\dd u\otimes(\dd\log a+\theta_{\bar u})
   =\dd u\otimes\theta_u .
\]
The connection is flat, since $\dd\theta_u=0$, and it extends the Bott
partial connection, since $\nabla_{\Bott}(\dd u)=0$ and $\theta_u$
vanishes on $T_{\F}$. 

It remains to see that $\nabla_{\mathrm{aff}}$ is rational. Let $\omega$ be
a rational one-form defining $\F$ and write
$\nabla_{\mathrm{aff}}(\omega)=\omega\otimes\eta$. On $V$ we have
$\omega=g\,\dd u$ with $g$ meromorphic, and hence
$\eta=\theta_u+\dd\log g$ is meromorphic on $V$. Thus $\eta$ is a
meromorphic one-form on the complement of
$\operatorname{Sing}(\F)$, which has codimension at least two;
by Levi's extension theorem it is meromorphic on
$X$, hence rational. Therefore $\nabla_{\mathrm{aff}}$ is a flat rational
connection on $\CNF$ extending the Bott partial connection. Its dual is a
flat rational connection on $\NF$ extending the Bott connection, that is,
a singular transverse affine structure in the sense of \Cref{def:ta}, and
$\F$ is transversely affine.
\end{proof}

\begin{remark}
The connection $\nabla_{\mathrm{aff}}$ has an intrinsic description. The
kernel of $D_\Lambda$ is a rank-$m$ subbundle of $T^m$ mapped
isomorphically onto $T^{m-1}$ by truncation, and the resulting splitting of
\eqref{eq:transverse-jet-sequence} determines a flat rational connection
on $T^{m-1}$ extending $\partial^{m-1}$: the companion connection of the
transverse equation, see \cite[Lemma~5.1 and Theorem~5.2]{Fazoli2025a} and
\cite[Theorem~A]{Fazoli2025b}. In the frame of a local first integral $u$
it is the companion system of $D_{\Lambda,u}$, the transverse counterpart
of the classical theory of \cite[\S~I.4]{Deligne1970}. Its determinant is a
connection on $\det T^{m-1}\simeq(\CNF)^{\otimes q}$, the isomorphism
coming from \eqref{eq:transverse-jet-sequence}, with local connection form
$A_{m-1,u}(u)\,\dd u$, and $\nabla_{\mathrm{aff}}$ is its $q$-th root. The
local description through the subprincipal coefficient appears already in
Pirio's thesis \cite[\S~5.2.5]{Pirio2004}.
\end{remark}

\begin{corollary}\label{cor:rank-increase}
Let $X$ be a smooth projective variety, let $\W$ be a $k$-web, and let $\F$
be a foliation transversal to $\W$. If
\[
   \rank(\W\boxt\F)>\rank(\W),
\]
then $\F$ is transversely affine. Consequently, it admits a Liouvillian first
integral over $\C(X)$ in the sense of \Cref{def:global-liouvillian-fi}.
\end{corollary}

\begin{proof}
The projection
\[
   A(\W\boxt\F)_x\rightarrow A_{\F}(\W)_x
\]
is surjective by definition, and its kernel is naturally identified with
$A(\W)_x$: an abelian relation of $\W\boxt\F$ with vanishing $\F$-component
is an abelian relation of $\W$. Hence
\[
   \rank(\W\boxt\F)=\rank(\W)+\dim_{\C}A_{\F}(\W)_x,
\]
and the rank inequality implies that $A_{\F}(\W)_x\neq0$. The exact sequence
\eqref{eq:component-exact} then gives $\Sol_{\F}(\W)_x\neq\C$, and $\F$ is
transversely affine by \Cref{thm:abel-to-affine}. The Liouvillian assertion
follows from the elementary implication of \Cref{thm:singer-criterion},
established right after its statement.
\end{proof}

\section{Foliations without rational first integrals}
\label{sec:no-rational-first-integrals}

\begin{theorem}
\label{thm:no-rfi}
Let $X$ be a smooth projective variety, let $\W$ be a
$k$-web, and let $\F$ be a foliation transversal to $\W$.
Suppose that $\F$ has no nonconstant rational first
integral and that $\Sol_{\F}(\W)\neq\C$. Then $\F$ is
transversely affine and every element of $\Sol_{\F}(\W)$
is defined over a Liouvillian extension of $\C(X)$.
\end{theorem}

Transverse affineness follows from \Cref{thm:abel-to-affine}.
To prove the assertion about the components, we first
record an elementary fact.

\begin{lemma}
\label{lem:no-new-first-integrals}
Let $K=\C(X)$, let $K'/K$ be a differential field extension,
and let $K''/K'$ be a finite algebraic extension.
If $\F$ has no nonconstant first integral in $K'$, then
it has no nonconstant first integral in $K''$.
\end{lemma}

\begin{proof}
Let $h\in K''$ be a first integral of $\F$.
For every rational vector field $\delta$ tangent to $\F$, acting on $K'$
and $K''$, we have $\delta(h)=0$. Differentiate the monic minimal
polynomial of $h$ over $K'$ with respect to $\delta$. The resulting
polynomial has smaller degree and vanishes at $h$, so it is zero.
Thus every coefficient of the minimal polynomial is annihilated by
every tangent derivation, and is a first integral in $K'$.
By hypothesis these coefficients belong to $\C$. Since $\C$ is
algebraically closed, $h\in\C$.
\end{proof}

\subsection{Closed defining forms}

\begin{proposition}
\label{prop:closed-extension-normal-form}
Let $X$ be a smooth projective variety, let $\W$ be a
$k$-web, and let $\F$ be a foliation transversal to $\W$.
Set $K=\C(X)$ and let $L/K$ be a Liouvillian extension.
Assume that $\F$ is defined by a nonzero closed one-form
$\alpha$ over $L$ and has no nonconstant first integral
in $L$. If $\Sol_{\F}(\W)\neq\C$, then there are
$2\leq m\leq k(k+1)/2$ and constants
$c_1,\ldots,c_{m-1}\in\C$ such that
\[
   \Sol_{\F}(\W)\subset
   \ker\left(
      \partial_t^m+c_{m-1}\partial_t^{m-1}
      +\cdots+c_1\partial_t
   \right),
\]
where $t$ is a foliated coordinate with $\dd t=\alpha$.
Consequently, every element of $\Sol_{\F}(\W)$ has the form
\[
   G(t)=\sum_{j=1}^r P_j(t)e^{\mu_jt},
   \qquad P_j\in\C[t],
\]
where the $\mu_j$ are the distinct roots of the
characteristic polynomial and $\deg P_j$ is smaller
than the multiplicity of $\mu_j$. In particular,
\[
   \Sol_{\F}(\W)\subset
   L\bigl(t,e^{\mu_1t},\ldots,e^{\mu_rt}\bigr),
\]
and this extension is Liouvillian over $K$.
\end{proposition}

\begin{proof}
Choose a vector field $w$ over $L$ with $\alpha(w)=1$.
Since $\alpha$ is closed, Cartan's formula gives
$\mathcal L_w\alpha=0$, so $w$ is an infinitesimal
automorphism of $\F$. By
\Cref{lem:rational-transversality}\textup{(iii)}, the
normalized minimal annihilator is
\[
   \mathcal D_w
   =w^m+c_{m-1}w^{m-1}+\cdots+c_1w,
\]
where the $c_\ell\in L$ are first integrals of $\F$.
Thus $c_\ell\in\C$.

As $w(t)=1$, the action of $w$ on local first integrals
is $\partial_t$. This gives the stated differential
equation, whose solutions are the exponential polynomials
above. Finally, adjoining $t$ is a primitive extension,
and
\[
   \frac{\dd e^{\mu_jt}}{e^{\mu_jt}}=\mu_j\alpha,
\]
so adjoining the exponentials gives a Liouvillian extension.
\end{proof}

\subsection{The general case}

Let $(\omega,\eta)$ be a rational transverse affine
structure for $\F$, so that
\[
   \dd\omega=\omega\wedge\eta,
   \qquad \dd\eta=0.
\]
Set $K=\C(X)$ and choose a nonzero germ
$E=\exp(\int\eta)$. Then $E\omega$ is closed and is
defined over the Liouvillian extension $K_1=K(E)$.

\begin{lemma}
\label{lem:K1-dichotomy}
If $K_1$ contains a nonconstant first integral of $\F$,
then either
\begin{enumerate}[label=\textup{(\roman*)}]
\item $\F$ has a nonconstant rational first integral; or
\item $\F$ is defined by a nonzero closed one-form over
a finite algebraic extension of $K$.
\end{enumerate}
\end{lemma}

\begin{proof}
For every rational vector field $\delta$ tangent to $\F$, set
$r_\delta=\eta(\delta)$, so that $\delta(E)=r_\delta E$.
If $r_\delta=0$ for every tangent $\delta$, then
$\eta\wedge\omega=0$ and $\omega$ is closed.
If $E$ is algebraic over $K$, then $E\omega$ gives
\textup{(ii)}. We may therefore assume that $r_\delta\neq0$
for some tangent $\delta$ and that $E$ is transcendental over $K$.

Let $h\in K(E)$ be a nonconstant first integral.
Extend every tangent derivation to $\overline K(E)$, and fix
the partial-fraction expansion
\[
   h=\sum_{j=-q}^{d}a_jE^j+
   \sum_{i=1}^{N}\sum_{s=1}^{m_i}
      \frac{b_{i,s}}{(E-e_i)^s},
\]
where $a_j,b_{i,s}\in\overline K$, the
$e_i\in\overline K^*$ are distinct, and
$b_{i,m_i}\neq0$.

If a nonzero pole $e_i$ occurs, apply each tangent derivation
$\delta$ to this expansion. The coefficient of
$(E-e_i)^{-(m_i+1)}$ in $\delta(h)=0$ gives
$\delta(e_i)=r_\delta e_i$ for every such $\delta$. Hence
\[
   \beta=\eta-\frac{\dd e_i}{e_i}
\]
is closed and annihilates all of $T_{\F}$, so
$\beta\wedge\omega=0$.
If $\beta\neq0$, it defines $\F$ over $K(e_i)$.
If $\beta=0$, then $e_i\omega$ is a nonzero closed
one-form defining $\F$.

If there are no nonzero poles, then
$h=\sum_{j=-q}^{d}a_jE^j$, and applying every tangent
$\delta$ gives
\[
   \delta(a_j)+j r_\delta a_j=0
   \qquad\text{for every tangent }\delta.
\]
Suppose that $a_j\neq0$ for some $j\neq0$. The closed form
\[
   \gamma=\frac{\dd a_j}{a_j}+j\eta
\]
annihilates all of $T_{\F}$, so $\gamma\wedge\omega=0$. If $\gamma\neq0$,
it gives \textup{(ii)}. Otherwise, choose $b$ with
$b^j=a_j^{-1}$. Then $\dd b/b=\eta$, so $b\omega$
is a nonzero closed defining form over a finite
extension of $K$.

The remaining case is $h=a_0$. Since $E$ is transcendental,
$\overline K\cap K(E)=K$, so $h\in K$, proving \textup{(i)}.
\end{proof}

\begin{proof}[Proof of \Cref{thm:no-rfi}]
By \Cref{thm:abel-to-affine}, $\F$ admits a rational
transverse affine structure $(\omega,\eta)$. With the
notation above, $K_1/K$ is Liouvillian and
$\alpha=E\omega$ is a nonzero closed defining form.

If $\F$ has no nonconstant first integral in $K_1$,
apply \Cref{prop:closed-extension-normal-form} with
$L=K_1$ and $\alpha=E\omega$.

Otherwise, \Cref{lem:K1-dichotomy} and the hypothesis
on $\F$ give a nonzero closed defining form over a
finite algebraic extension $K'/K$. By
\Cref{lem:no-new-first-integrals}, $\F$ still has no
nonconstant first integral in $K'$. Since $K'/K$ is
Liouvillian, \Cref{prop:closed-extension-normal-form}
applies with $L=K'$. 

In both cases, all elements of
$\Sol_{\F}(\W)$ belong to a Liouvillian extension of $K$.
\end{proof}
\section{Differential modules generated by abelian relations}
\label{sec:component-modules}

We associate a differential module to each component of a
functional abelian relation and study its simple factors
when the corresponding foliation has no nonconstant rational
first integral.

\subsection{Differential modules} We start by defining differential modules over a differential field and its composition series.

\begin{definition}
\label{dfn:differential-module}
An \emph{integrable differential module} over $K$ is a
finite-dimensional $K$-vector space $M$ with a flat
connection $\nabla$. All differential modules considered
below are integrable. We usually omit $\nabla$ and denote
the action of $\partial\in\Delta$ on $m\in M$ by
$\partial(m)$. Morphisms, differential submodules,
quotients, and finite direct sums are defined in the
usual way.
\end{definition}

\begin{definition}
A \emph{composition series} for a differential module $M$
is a flag of differential submodules
\[
   0=M_0\subset M_1\subset\cdots\subset M_r=M
\]
such that each quotient $M_i/M_{i-1}$ is simple.
These quotients are called the \emph{simple factors}
of $M$.
\end{definition}

By the Jordan--H\"older theorem, the multiset of simple
factors is independent, up to isomorphism, of the
composition series; see
\cite[Chapter~2]{vanDerPutSinger2003}.
For a short exact sequence
\[
   0\longrightarrow M'\longrightarrow M
   \longrightarrow M''\longrightarrow0,
\]
the simple factors of $M$ are those of $M'$ and $M''$,
with multiplicities. In particular, the simple factors
of a submodule or quotient occur among those of $M$,
and the simple factors of a finite direct sum are those
of its summands.

\begin{definition}
A differential module $M$ \emph{becomes trivial} over a
differential extension $K'/K$ if $K'\otimes_K M$ has a
horizontal basis.
\end{definition}

\subsection{The differential module generated by a component}

Set $K=\C(X)$ and let $u$ be an element of a differential
extension of $K$ inside $\mathcal M_{X,x}$.
We define $E(u)$ to be the smallest $K$-vector subspace
containing $u$ and stable under all rational derivations.
For the commuting derivations $\partial_1,\ldots,\partial_n$
determined by a transcendence basis of $K$, we have
\begin{equation}
\label{eq:def-component-module}
   E(u)=\operatorname{Span}_K
   \left\{\partial_1^{a_1}\cdots\partial_n^{a_n}(u)
   \mid a_1,\ldots,a_n\geq0\right\}.
\end{equation}
When finite-dimensional, $E(u)$ is the smallest
differential $K$-module containing $u$.

\begin{proposition}
\label{prop:finite-component-module-main}
Let $\W$ be a $k$-web on a smooth projective variety,
and let $u$ be a component of a functional abelian
relation of $\W$. Then $E(u)$ is finite-dimensional
over $K$.
\end{proposition}

\begin{proof}
Let $\pi:X'\to X$ be a decomposing cover and set
$K'=\C(X')$. Since $E_K(u)\subset E_{K'}(u)$ and $K'/K$
is finite, it suffices to prove that $E_{K'}(u)$ is
finite-dimensional over $K'$. We may therefore replace
$K$ by $K'$ and assume that $\W$ is decomposable and
that $u$ is a component along a foliation $\F$.

If $u$ is constant, then $E(u)=K\cdot u$.
Otherwise, \Cref{lem:rational-existence}, applied to
$\F$ and the $(k-1)$-web formed by the remaining factors,
gives a nonzero rational annihilator $\Lambda$ with
$\Lambda(j^N u)=0$.

Choose a $K$-basis $v_1,\ldots,v_{n-1},w$ of $\Der_{\C}(K)$,
with $v_1,\ldots,v_{n-1}$ tangent to $\F$ and $w$ transverse to $\F$. At a general point, the map
\[
   j^N f\longmapsto(f,w(f),\ldots,w^N(f))
\]
is related to the frame of a foliated coordinate by a
triangular change of frame with nonzero diagonal.
Thus $\Lambda$ can be written in powers of $w$ with
coefficients in $K$, and its equation gives
\begin{equation}
\label{eq:component-recurrence-main}
   w^m(u)=a_{m-1}w^{m-1}(u)+\cdots+a_0u,
   \qquad a_i\in K.
\end{equation}
It follows that
\[
   V=\operatorname{Span}_K
   \{u,w(u),\ldots,w^{m-1}(u)\}
\]
is stable under $w$.
Write
\[
   [v_a,w]=\sum_{b=1}^{n-1}a_{ab}v_b+b_a w,
   \qquad a_{ab},b_a\in K.
\]
Starting from $v_a(u)=0$ and using
\[
   v_a(w^{r+1}u)=w\bigl(v_a(w^r u)\bigr)+[v_a,w](w^r u),
\]
a simultaneous induction for all $a$ gives
\[
   v_a(w^r u)\in\operatorname{Span}_K\{u,w(u),\ldots,w^r(u)\}.
\]
Thus $V$ is stable under every $v_a$ and under $w$.
Hence $V$ is a differential module containing $u$,
and therefore $V=E(u)$.
\end{proof}

\subsection{The case without rational first integrals}

We first record a consequence of the Galois isomorphism
$L\otimes_K L\simeq\prod_{\sigma\in G}L$; see
\cite[Chapter~V, \S10]{BourbakiAlgebraII}.
Geometrically, for a Galois covering $\pi:X'\to X$,
it corresponds to the identity
\[
   \pi^*\pi_*\mathcal V
   \simeq\bigoplus_{\sigma\in G}\sigma^*\mathcal V
\]
for flat connections over the locus where $\pi$ is
\'etale. 

\begin{lemma}
\label{lem:restriction-rank-one-flags}
Let $L/K$ be a finite Galois extension of differential
fields. Suppose that a differential $K$-module $M$
embeds in a differential $L$-module $V$ having a flag
with rank-one simple factors. Then $L\otimes_K M$
has a flag with rank-one simple factors.
\end{lemma}

\begin{proof}
Set $G=\operatorname{Gal}(L/K)$ and regard $V$ as a
differential $K$-module by restriction of scalars.
The idempotents of
$L\otimes_K L\simeq\prod_{\sigma\in G}L$
are horizontal: differentiating $e^2=e$ gives
$(1-2e)\partial(e)=0$, and $1-2e$ is a unit. Hence
\[
   L\otimes_K V
   \simeq\bigoplus_{\sigma\in G}{}^\sigma V
\]
as differential $L$-modules, where ${}^\sigma V$
denotes $V$ with scalar multiplication twisted by
$\sigma$. Each summand has a flag with rank-one
quotients, so the direct sum has a composition series
with rank-one factors. Intersecting this series with
$L\otimes_K M$ and omitting repetitions proves the
claim.
\end{proof}

\begin{lemma}
\label{lem:exp-monomial-independence}
Let $L\subset\mathcal M_{X,x}$ be a differential extension
of $K=\C(X)$ in which $\F$ has no nonconstant first integral.
Let $t$ be a nonconstant holomorphic first integral of $\F$
at $x$. Then
\[
   \bigl\{t^a e^{\nu t}\mid a\in\N,\ \nu\in\C\bigr\}
\]
is linearly independent over $L$.
\end{lemma}

\begin{proof}
Suppose that
\[
   \sum_{i=1}^N \ell_i t^{a_i}e^{\nu_i t}=0
\]
is a nontrivial relation over $L$, with distinct pairs
$(a_i,\nu_i)$ and $N$ minimal. Dividing by $\ell_1$,
we may assume that $\ell_1=1$. Applying any rational vector
field $v$ tangent to $\F$ gives
\[
   \sum_{i=2}^N v(\ell_i)t^{a_i}e^{\nu_i t}=0.
\]
By minimality, $v(\ell_i)=0$ for every $i$ and every tangent $v$.
Thus the $\ell_i$ are first integrals in $L$, so they belong to $\C$. Since $t$ is nonconstant,
this contradicts the linear independence of the
functions $s^a e^{\nu s}$ over $\C$.
\end{proof}

\begin{lemma}
\label{lem:dilation-exponential-polynomial}
Let $G$ be a germ at $0$ of a holomorphic function of one variable.
If the $\C$-linear span of
$\{G(ct)\mid c\in\C^*\}$ is finite-dimensional,
then $G$ is a polynomial.
\end{lemma} 

\begin{proof}
Set $V=\operatorname{Span}_{\C}\{G(ct)\mid c\in\C^*\}$
and write $G(t)=\sum_{n\geq0}a_nt^n$.
For each $n$ with $a_n\neq0$, let $\ell_n\in V^\vee$
be the functional extracting the coefficient of $t^n$.
These functionals are linearly independent: evaluating
a finite relation $\sum_n b_n\ell_n=0$ on $G(ct)$ gives
\[
   \sum_n b_na_nc^n=0
   \qquad\text{for every }c\in\C^*,
\]
hence every $b_n$ is zero. Since $V^\vee$ is
finite-dimensional, only finitely many $a_n$ are nonzero.
\end{proof}

\begin{lemma}
\label{lem:dilation-invariance}
Let $(\omega,\eta)$ be a rational transverse affine
structure for $\F$, and let
$\rho,t\in\mathcal M_{X,x}$ satisfy
\[
   \dd\rho=\rho\eta,
   \qquad
   \dd t=\rho\omega.
\]
Assume that $\rho$ is transcendental over $K$, that $t$
is transcendental over $K(\rho)$, and that $\F$ has no
nonconstant first integral in $K(\rho)$. Let
\[
   u=G(t)=\sum_{j=1}^r P_j(t)e^{\mu_jt},
   \qquad P_j\in\C[t],\quad
   \mu_j\in\C\ \text{distinct},
\]
and suppose that $\Lambda(u)=0$, where
\[
   \Lambda=w^m+a_{m-1}w^{m-1}+\cdots+a_0,
   \qquad a_i\in K,
\]
and $w$ is a rational vector field with $\omega(w)=1$.
Then $G$ is a polynomial.
\end{lemma}

\begin{proof}
Fix $c\in\C^*$. Since $\rho$ is transcendental over $K$
and $t$ is transcendental over $K(\rho)$, the assignments
\[
   \rho\longmapsto c\rho,
   \qquad t\longmapsto ct
\]
define a differential $K$-automorphism $\phi$ of
$K(\rho,t)$. By \Cref{lem:exp-monomial-independence},
applied with $L=K(\rho)$, the exponentials $e^{\nu t}$
are linearly independent over $K(\rho,t)$, as one sees
by clearing denominators. Thus the ring they generate
is the group algebra
\[
   \bigoplus_{\nu\in\C}K(\rho,t)e^{\nu t},
\]
and $\phi$ extends to a differential embedding of this
ring by
\[
   e^{\nu t}\longmapsto e^{\nu ct}.
\]
Indeed, the identities $\dd\rho=\rho\eta$ and
$\dd t=\rho\omega$ show that this extension commutes
with the derivations.

Since $\Lambda$ has coefficients in $K$, it commutes
with $\phi$. Hence
\[
   \Lambda(G(ct))=\phi(\Lambda(G(t)))=0
   \qquad\text{for every }c\in\C^*.
\]
The first integrals annihilated by $\Lambda$ form a
$\C$-vector space of dimension at most $m$: restricting
to a general local transversal gives an ordinary
differential equation of order $m$.
Therefore the span of the functions $G(ct)$ is
finite-dimensional, and
\Cref{lem:dilation-exponential-polynomial} implies
that $G$ is a polynomial.
\end{proof}

\begin{proposition}
\label{prop:no-rfi-module-factors}
Let $u$ be a component of a functional abelian relation
of a web $\W$ on a smooth projective variety.
Suppose that the corresponding foliation $\F$ has
no nonconstant rational first integral.
Then there is a finite algebraic extension $K'/K$
such that $K'\otimes_K E(u)$ has a composition series
with rank-one simple factors.
\end{proposition}

\begin{proof}
If $u$ is constant, then $E(u)=K\cdot u$ and the
statement is immediate. Assume that $u$ is nonconstant.
We distinguish two cases, according to the proof of
\Cref{thm:no-rfi}.

First suppose that $\F$ is defined by a nonzero closed
one-form $\alpha$ over a finite algebraic extension
$L/K$. Enlarging $L$ to its Galois closure, we may
assume that $L/K$ is Galois. By
\Cref{lem:no-new-first-integrals}, $\F$ has no
nonconstant first integral in $L$.
Choose a foliated coordinate $t$ with $\dd t=\alpha$.
By \Cref{prop:closed-extension-normal-form},
\[
   u=\sum_{j=1}^r P_j(t)e^{\mu_jt}.
\]
Omitting zero terms, consider the $L$-vector space
\[
   V=\bigoplus_{j=1}^r
     \bigoplus_{a=0}^{\deg P_j}
     L\cdot t^ae^{\mu_jt}.
\]
The sum is direct by
\Cref{lem:exp-monomial-independence}.
For every $\partial\in\Delta$,
\[
   \partial\bigl(t^ae^{\mu_jt}\bigr)
   =\alpha(\partial)
      \bigl(a\,t^{a-1}+\mu_jt^a\bigr)e^{\mu_jt}.
\]
Ordering the monomials by the lexicographic order
on $(j,a)$, their partial spans form a flag of
differential $L$-submodules with rank-one quotients.
Since $V$ contains $u$, we have $E(u)\subset V$.
Applying \Cref{lem:restriction-rank-one-flags} with
$M=E(u)$ proves the proposition with $K'=L$.

Suppose now that no finite algebraic extension of $K$
admits a nonzero closed one-form defining $\F$.
In the notation of the proof of \Cref{thm:no-rfi},
$\dd\omega\neq0$ and $\rho=E$ is transcendental over
$K$: otherwise $\omega$, or $E\omega$ over $K(E)$,
would be such a form. By \Cref{lem:K1-dichotomy},
$\F$ has no nonconstant first integral in
$K_1=K(\rho)$.

We have $\dd\rho=\rho\eta$ and $\dd t=\rho\omega$
for a foliated coordinate $t$. Applying
\Cref{prop:closed-extension-normal-form} with
$L=K_1$ and $\alpha=\rho\omega$, we obtain
\[
   u=G(t)=\sum_{j=1}^r P_j(t)e^{\mu_jt},
   \qquad P_j\in\C[t].
\]
Moreover, $t$ is transcendental over $K_1$:
otherwise it would be a nonconstant first integral
in a finite extension of $K_1$, contrary to
\Cref{lem:no-new-first-integrals}.
As in the proof of
\Cref{prop:finite-component-module-main}, write a
rational annihilator of $u$ in powers of a rational
vector field $w$ with $\omega(w)=1$.
Then \Cref{lem:dilation-invariance} applies, so $G$
is a polynomial, say of degree $q$.

Consider the $K$-vector space
\[
   W=\bigoplus_{\substack{a,b\geq0\\a+b\leq q}}
      K\,\rho^at^b\subset K(\rho,t).
\]
The sum is direct because $\rho$ and $t$ are
algebraically independent over $K$. The identity
\[
   \dd(\rho^at^b)
   =a\,\rho^at^b\eta
     +b\,\rho^{a+1}t^{b-1}\omega
\]
shows that $W$ is a differential $K$-module.
Order its monomials by the lexicographic order on
$(a+b,b)$. Their partial spans form a flag of
differential submodules with rank-one quotients.

Since $u=G(t)\in W$, we have $E(u)\subset W$.
Intersecting this flag with $E(u)$ and omitting
repetitions gives a composition series with rank-one
simple factors. Thus $K'=K$ suffices in this case.
\end{proof}
\section{Resonance and descent through rational fibrations}
\label{sec:resonance-descent}

We investigate two structural properties of the differential modules $E(u_i)$ associated with an abelian relation $u_1+ \cdots + u_k =0$ on a smooth projective variety. The first is \emph{resonance}: every simple factor $E(u_i)$ occurs, up to isomorphism, as the simple factor of  $E(u_j)$ for some $j\neq i$. The second is \emph{descent} through rational fibrations, saying that when $u_i$ is a component associated with a rational fibration, $E(u_i)$ is the pullback of a differential module over the base of that fibration.

\subsection{Resonance}

Let
\[
   u_1+\cdots+u_k=0
\]
be a functional abelian relation of a $k$-web $\W$ on
a smooth projective variety $X$. Set $K=\C(X)$ and
$E_i=E(u_i)$, regarding these modules as subspaces of
the same field of meromorphic germs. They are
finite-dimensional by
\Cref{prop:finite-component-module-main}.

\begin{lemma}
\label{lem:component-resonance-main}
Every simple factor of $E_i$ is isomorphic to a simple
factor of some $E_j$ with $j\neq i$.
\end{lemma}

\begin{proof}
Since $u_i=-\sum_{j\neq i}u_j$ and the sum of the
modules $E_j$ is stable under differentiation,
\[
   E_i\subseteq\sum_{j\neq i}E_j.
\]
The right-hand side is a quotient of
$\bigoplus_{j\neq i}E_j$, so the claim follows from
the Jordan--H\"older theorem.
\end{proof}

\subsection{Modules pulled back from a curve}

Let $p:X\dashrightarrow B$ be a rational fibration
from a smooth projective variety to a curve.
After resolving indeterminacies and taking the Stein
factorization, we may assume that $p$ is a morphism
with connected generic fibre. Set $F=\C(B)$ and
identify $F$ with its image in $K=\C(X)$.
There is an exact sequence
\[
   0\longrightarrow\Der_F(K)
   \longrightarrow\Der_{\C}(K)
   \longrightarrow K\otimes_F\Der_{\C}(F)
   \longrightarrow0.
\]
The elements of $\Der_F(K)$ are the rational vector
fields tangent to the fibres of $p$.

\begin{lemma}
\label{lem:vertical-constants}
The common field of constants of the vertical derivations is
\[
   K^{\Der_F(K)}
   :=\{h\in K\mid v(h)=0\text{ for every }v\in\Der_F(K)\}
   =F.
\]
\end{lemma}

\begin{proof}
If $h\in K^{\Der_F(K)}$, then $\dd_{K/F}h=0$.
In characteristic zero this means that $h$ is algebraic over $F$.
Since the generic fibre is connected, $F$ is algebraically closed in $K$,
and hence $h\in F$. The reverse inclusion is immediate.
\end{proof}

\begin{lemma}
\label{lem:vertical-stable-subspaces}
Let $d\geq1$, and let the vertical derivations $\Der_F(K)$ act
componentwise on the $K$-vector space $K^d$ by differentiation of the
coefficients. Every $K$-vector subspace $N\subset K^d$ stable under
all these derivations has the form $N=K\otimes_F N_F$ for an
$F$-vector subspace $N_F\subset F^d$.
\end{lemma} 

\begin{proof}
Let $r=\dim_K N$ and choose a matrix
$A\in\operatorname{Mat}_{d\times r}(K)$ whose columns
form a basis of $N$.
For every $v\in\Der_F(K)$, stability gives
\[
   v(A)=AC_v
\]
for some $C_v\in\operatorname{Mat}_r(K)$.
Its Pl\"ucker coordinates therefore satisfy
\[
   v(p_I(A))=\operatorname{tr}(C_v)\,p_I(A).
\]
Choose $J$ with $p_J(A)\neq0$. Then
\[
   v\left(\frac{p_I(A)}{p_J(A)}\right)=0
   \qquad\text{for every }v\in\Der_F(K),
\]
so every ratio belongs to $F$ by
\Cref{lem:vertical-constants}. Thus $[N]$ is an
$F$-rational point of the Grassmannian, and the
corresponding subspace $N_F\subset F^{d}$ satisfies
$N=K\otimes_F N_F$. 
\end{proof} 

For a differential $F$-module $M_F$, the pullback
$K\otimes_F M_F$ has the connection
\[
   \delta(f\otimes m)
   =\delta(f)\otimes m+f\,\delta(m),
   \qquad
   \delta\in\Der_{\C}(K),
\]
where the action on $m$ is induced by
$\Der_{\C}(K)\to K\otimes_F\Der_{\C}(F)$.
This defines a functor
\[
   K\otimes_F-:
   \operatorname{DiffMod}(F)
   \longrightarrow\operatorname{DiffMod}(K).
\]

\begin{lemma}
\label{lem:descent-subquotients-main}
Let $M_F$ be a differential $F$-module and let
$N\subset K\otimes_F M_F$ be a differential
$K$-submodule. Then there is a differential
$F$-submodule $N_F\subset M_F$ such that
$N=K\otimes_F N_F$.
Consequently, the essential image of $K\otimes_F-$
is closed under subobjects, quotients, and subquotients.
\end{lemma}

\begin{proof}
Choose an $F$-basis of $M_F$, identifying its pullback with $K^d$,
where $d=\dim_F M_F$. Every vertical derivation acts on $K^d$ by
differentiation of the coefficients. Since $N$ is stable under all
vertical derivations, \Cref{lem:vertical-stable-subspaces} gives an
$F$-subspace $N_F\subset M_F$ with $N=K\otimes_F N_F$.

Let $\partial\in\Der_{\C}(F)$ and choose a rational lift
$\delta\in\Der_{\C}(K)$. For $m\in M_F$,
\[
   \delta(1\otimes m)=1\otimes\partial(m).
\]
Since $N$ is a differential submodule, this implies
\[
   \partial(N_F)\subset
   M_F\cap(K\otimes_F N_F)=N_F.
\]
Thus $N_F$ is a differential $F$-submodule.
Finally,
\[
   (K\otimes_F M_F)/(K\otimes_F N_F)
   \simeq K\otimes_F(M_F/N_F),
\]
which proves the assertion about quotients and
subquotients.
\end{proof}

\begin{proposition}
\label{prop:algebraic-component-descent-main}
Let $\F$ be a foliation on a smooth projective variety,
defined by a rational fibration $p:X\dashrightarrow B$
with connected generic fibre. Let $u$ be the
$\F$-component of a functional abelian relation of
$\W\boxt\F$, where $\W$ is transversal to $\F$.
Then $E(u)$ is pulled back from a differential module
over $F=\C(B)$.
\end{proposition}

\begin{proof}
If $u$ is constant, then $E(u)=K\cdot u$ is the
pullback of $F\cdot u$. We may therefore assume that
$u$ is nonconstant.

Choose a nonconstant $t\in F$ and let
$\delta\in\Der_{\C}(K)$ be a rational lift of
$\partial_t$. For every vertical vector
field $v$, the bracket $[\delta,v]$ is vertical.
Thus $\delta$ is an infinitesimal automorphism of
$\F$, generically transverse to $\F$. Let $D$ be the normalized minimal annihilator of
$\Sol_{\F}(\W)$. By
\Cref{lem:rational-transversality}\textup{(iii)},
applied with $w=\delta$ and $L=K$,
\[
   D=\delta^m+\sum_{\ell=1}^{m-1}c_\ell\delta^\ell,
\]
where the $c_\ell$ are rational first integrals.
They are annihilated by all vertical derivations, so
$c_\ell\in F$ by \Cref{lem:vertical-constants}.

Let $M_D=\bigoplus_{\ell=0}^{m-1}F e_\ell$ be the
companion differential module, with
\[
   \partial_t(e_\ell)=e_{\ell+1}
   \quad(0\leq\ell\leq m-2),
   \qquad
   \partial_t(e_{m-1})
   =-\sum_{\ell=1}^{m-1}c_\ell e_\ell.
\]
Consider the $K$-linear map
\[
   \psi:K\otimes_F M_D\longrightarrow E(u),
   \qquad
   e_\ell\longmapsto\delta^\ell(u).
\]
Every $\delta^\ell(u)$ is a first integral of $\F$,
so $\psi$ commutes with vertical derivations.
It also commutes with $\delta$, by the definition of
$M_D$ and the equation $D(u)=0$.
Since $\delta$ and the vertical vector fields span $\Der_{\C}(K)$,
the map $\psi$ is a morphism of differential modules.

Its image is a differential submodule containing $u$,
and therefore equals $E(u)$. Thus $E(u)$ is a quotient
of $K\otimes_F M_D$, and
\Cref{lem:descent-subquotients-main} gives the conclusion.
\end{proof}

\subsection{Two transverse fibrations}

\begin{lemma}
\label{lem:finite-transverse-base-change}
Let $p:X\dashrightarrow B$ be a rational fibration
with connected generic fibre from a smooth projective
variety to a curve. Set $F=\C(B)$ and $K=\C(X)$.
Let $L/F$ be a finite differential field extension,
and form $K'=KL$ in a common overfield.
If a differential module $M_F$ becomes trivial over
$L$, then $K\otimes_F M_F$ becomes trivial over $K'$.
\end{lemma}

\begin{proof}
The extension $K'/K$ is finite. Since $F$ is
algebraically closed in $K$, the fields $K$ and $L$
are linearly disjoint over $F$, and
\[
   K\otimes_F L\simeq K'.
\]
Choose a nonzero derivation $\partial$ of $F$ and
denote its extension to $L$ by $\partial_L$.
For $\delta\in\Der_{\C}(K)$, write
$\delta|_F=c_\delta\partial$ with $c_\delta\in K$.
The extension of $\delta$ to $K'$ restricts on $L$
to $c_\delta\partial_L$. Hence a horizontal basis of
$L\otimes_F M_F$ remains horizontal after extending
scalars to $K'$.
\end{proof}

\begin{theorem}
\label{thm:transverse-descent-main}
Let $X$ be a smooth projective variety, and let
$p:X\dashrightarrow B$ and $q:X\dashrightarrow C$
be rational fibrations with connected generic fibres
defining distinct foliations.
If a differential module $S$ over $K=\C(X)$ is
pulled back both from $F=\C(B)$ and from $G=\C(C)$,
then $S$ becomes trivial after a finite algebraic
extension of $K$.
\end{theorem}

\begin{proof}
Choose differential modules $S_F$ and $S_G$ and
an isomorphism
\[
   \Phi:K\otimes_F S_F
   \longrightarrow K\otimes_G S_G
\]
of differential $K$-modules.
Fix bases of $S_F$ and $S_G$, write their connections as
$\nabla_F=\dd+\Omega_F$ and $\nabla_G=\dd+\Omega_G$, and let
$A\in\operatorname{GL}_r(K)$ be the matrix of $\Phi$.
Here $\Omega_F$ and $\Omega_G$ are matrices of rational one-forms on
$B$ and $C$, respectively. Compatibility with the connections gives
\[
   \dd A=A\,p^*\Omega_F-q^*\Omega_G\,A.
\]

After resolving the indeterminacies of $p$ and $q$, choose a general
smooth irreducible closed fibre $Y=q^{-1}(c)$ such that $p|_Y$ is
nonconstant, $c$ lies outside the poles of $\Omega_G$, and $Y$ is not
contained in the polar divisors of the entries of $A$ or $A^{-1}$.
Such a choice is possible because $p$ and $q$ define distinct foliations.
Set $H=\C(Y)$ and regard $F$ as a subfield of $H$ through $p|_Y$.
Restricting the preceding identity to $Y$ gives
\[
   \dd\bar A=\bar A\,(p|_Y)^*\Omega_F,
   \qquad \bar A\in\operatorname{GL}_r(H).
\]
Passing to relative differentials over $F$, we obtain
$\dd_{H/F}\bar A=0$. In characteristic zero, every entry of
$\bar A$ is therefore algebraic over $F$. Thus
\[
   F'=F(\text{entries of }\bar A)\subset H
\]
is a finite extension of $F$, and $\bar A\in\operatorname{GL}_r(F')$.

The natural map
$H\otimes_{F'}\Omega^1_{F'/\C}\to\Omega^1_{H/\C}$ is injective,
so the restricted connection identity already holds over $F'$.
Consequently,
\[
   \dd(\bar A^{-1})=-\Omega_{F'}\,\bar A^{-1},
\]
where $\Omega_{F'}$ is the pullback of $\Omega_F$ to $F'$.
The columns of $\bar A^{-1}$ form a horizontal basis of
$F'\otimes_F S_F$. Embedding $F'$ over $F$ in an algebraic closure
of $K$, \Cref{lem:finite-transverse-base-change} shows that $S$
becomes trivial over the finite extension $KF'/K$.
\end{proof} 

\begin{corollary}
\label{cor:transverse-resonance-main}
A differential module which is simultaneously a
subquotient of modules pulled back through two
rational fibrations defining distinct foliations
becomes trivial after a finite algebraic extension
of $K$.
\end{corollary}

\begin{proof}
Apply \Cref{lem:descent-subquotients-main} to both
fibrations, followed by \Cref{thm:transverse-descent-main}.
\end{proof}
\section{Proof of the main theorem}
\label{sec:proof-main-theorem}

We now prove \Cref{thm:main}, restated below.

\begin{theorem}
\label{thm:liouvillian-abelian-relations}
Let $\W$ be a $k$-web on a smooth projective variety $X$.
Every germ of an abelian relation of $\W$ at a general
regular point is defined over a Liouvillian extension
of $\C(X)$, in the sense of
\Cref{def:liouvillian-abelian-relation}.
\end{theorem}

\begin{lemma}
\label{lem:rank-one-reduction}
Let $\W=\F_1\boxt\cdots\boxt\F_k$ be a decomposable
$k$-web on a smooth projective variety $X$, and let
\[
   u_1+\cdots+u_k=0
\]
be a functional abelian relation. Set $K=\C(X)$ and
$E_i=E(u_i)$. There is a finite algebraic extension
$K'/K$ such that every $K'\otimes_K E_i$ has a
composition series with rank-one simple factors.
\end{lemma}

\begin{proof}
Let $S$ be a simple factor of $E_i$. By
\Cref{lem:component-resonance-main}, it is also
isomorphic to a simple factor of some $E_j$ with
$j\neq i$.

If one of $\F_i,\F_j$ has no nonconstant rational
first integral, \Cref{prop:no-rfi-module-factors}
gives a finite extension over which the corresponding
component module has rank-one simple factors.
The same holds for its subquotient $S$.

If both foliations admit rational first integrals,
\Cref{prop:algebraic-component-descent-main} shows
that $E_i$ and $E_j$ are pulled back through rational
fibrations defining the distinct foliations
$\F_i$ and $\F_j$. Hence $S$ becomes trivial after
a finite algebraic extension, by
\Cref{cor:transverse-resonance-main}.

There are only finitely many simple factors among
the $E_i$. Take the compositum $K'$ of the finite
extensions obtained for these factors.
Scalar extension is exact and preserves rank-one
modules, so extending and refining a composition
series of each $E_i$ gives the required series.
\end{proof}

\begin{proof}[Proof of \Cref{thm:liouvillian-abelian-relations}]
By
\Cref{rmk:liouvillian-function-form,rmk:liouvillian-descent-covers},
it suffices to consider a functional abelian relation
\[
   u_1+\cdots+u_k=0
\]
of a decomposable web
$\W=\F_1\boxt\cdots\boxt\F_k$.
Set $K=\C(X)$ and $E_i=E(u_i)$, and choose the finite
extension $K'/K$ given by \Cref{lem:rank-one-reduction}
inside $\mathcal M_{X,x}$.

Let $K'E_i\subset\mathcal M_{X,x}$ be the $K'$-linear
span of $E_i$. The multiplication map
\[
   K'\otimes_K E_i\longrightarrow K'E_i,
   \qquad a\otimes f\longmapsto af,
\]
is a surjective morphism of differential $K'$-modules.
Thus $K'E_i$ also has a composition series with
rank-one simple factors.

Fix $i$ and choose a basis $e_1,\ldots,e_r$ of $K'E_i$
adapted to such a series. Then
\begin{equation}
\label{eq:triangular-component-system}
   \dd e_j=e_j\omega_j+
   \sum_{\ell<j}e_\ell\omega_{j\ell},
\end{equation}
where $\omega_j,\omega_{j\ell}$ are one-forms over $K'$.
Flatness gives $\dd\omega_j=0$. 

We first consider the exponential adjunctions
\[
    z_j=\exp\left(\int\omega_j\right), 1\le j \le r,
\]
over $K'$.
Equation \eqref{eq:triangular-component-system} gives
\[
   \dd\left(\frac{e_j}{z_j}\right)
   =\sum_{\ell<j}\frac{e_\ell}{z_j}\omega_{j\ell}.
\]
For $j=1$, the right-hand side is zero, so $e_1$
is a constant multiple of $z_1$. Inductively, once
$e_1,\ldots,e_{j-1}$ have been adjoined, adjoining
$e_j/z_j$ is a primitive extension. Hence all the
$e_j$, and therefore $u_i$, belong to a Liouvillian
extension of $K'$.

By \Cref{rmk:constants-automatic}, taking the compositum of these extensions for all
$i$ gives a Liouvillian extension containing every
component of an abelian relation of the web $\W$. Finally, \Cref{lem:liouvillian-correspondence}\textup{(i)} implies that this compositum is also Liouvillian over $K$.
\end{proof}

\bibliographystyle{plain}
\bibliography{references}

\end{document}